\documentclass[]{amsart}

\usepackage{pgf,tikz}
\usepackage{a4wide}
\usepackage{hyperref,enumerate}  
\usepackage{amssymb,amsopn,mathrsfs,amsmath,amsfonts,amsthm}
\usepackage{amsthm}
\usepackage{pgf,tikz}
\makeatletter
\newlength\knuthian@fdfive
\def\mathpal@save#1{\let\was@math@style=#1\relax}
\def\utilde#1{\mathpalette\mathpal@save
              {\setbox124=\hbox{$\was@math@style#1$}%
\setbox125=\hbox{$\fam=3\global\knuthian@fdfive=\fontdimen5\font$}
\setbox125=\hbox{$\widetilde{\vrule height 0pt depth 0pt width \wd124}$}%
               \baselineskip=1pt\relax
               \vtop{\copy124\copy125\vskip -\knuthian@fdfive}}}
\makeatother

\usetikzlibrary{arrows,%petri,
positioning,calc, %backgrounds,
decorations.pathmorphing,fit,matrix,decorations}

\newcommand{\I}{{[0,1]}}

\newcommand{\A}{\mathsf{A}}
\newcommand{\B}{\mathsf{B}}
\newcommand{\C}{\mathsf{C}}
\newcommand{\D}{\mathsf{D}}
\newcommand{\G}{\mathsf{G}}
\newcommand{\F}{\mathscr{F}}
\newcommand{\Z}{\ensuremath{\mathbb{Z}}}

\newcommand{\uline}[1]{\underline{#1}}

\newcommand{\twiddle}[1]{\utilde{#1}}    
       
\DeclareMathOperator{\FV}{\mathscr{V}}
\DeclareMathOperator{\McN}{\mathscr{M}}
\DeclareMathOperator{\FW}{\mathscr{W}}
\DeclareMathOperator{\FI}{\mathscr{I}}
\DeclareMathOperator{\Q}{\mathscr{Q}}
\DeclareMathOperator{\FC}{\mathscr{C}}
\DeclareMathOperator{\FST}{\mathscr{S}}
\DeclareMathOperator{\MAX}{\mathrm{Max}}

\DeclareMathOperator{\MVp}{\mathcal{MV}_{\mathtt{p}}}
\DeclareMathOperator{\MV}{\mathcal{MV}}
\DeclareMathOperator{\K}{\mathcal{K}}
\DeclareMathOperator{\MVS}{\mathcal{SMV}}
\DeclareMathOperator{\TSO}{\mathcal{T}}
\DeclareMathOperator{\TS}{\mathcal{T}_{c}}
\DeclareMathOperator{\Hom}{\mathsf{Hom}}

\DeclareMathOperator{\st}{\mathtt{st}}
\DeclareMathOperator{\ev}{\mathtt{ev}}
\DeclareMathOperator{\cev}{\mathtt{co-ev}}

\newtheorem{theorem}{Theorem}[section]
\newtheorem{lemma}[theorem]{Lemma}
\newtheorem{corollary}[theorem]{Corollary}
\newtheorem{proposition}[theorem]{Proposition}

\theoremstyle{definition}%per chi segue questa linea
\newtheorem{definition}[theorem]{Definition}
\newtheorem{remark}[theorem]{Remark}

\newcommand{\lang}{\mathcal{L}}
\newcommand{\seq}{\subseteq}
\newcommand{\FMV}{\F}
\newcommand{\word}[1]{\emph{#1}}
\DeclareMathOperator{\rad}{\mathtt{rad}}
\newcommand{\pX}{X}
\newcommand{\pY}{Y}
\renewcommand{\pm}{m}

\begin{document}

\title[MV-algebras and natural dualities]{MV-algebras, infinite dimensional polyhedra, and natural dualities}

\author{Leonardo M. Cabrer \and Luca Spada}

\address{Leonardo M.\ Cabrer:
Institute of Computer Languages,
Technische Universit\"at Wien,
Favoritenstrasse 9-11, A-1040 Wien, Austria. 
 }
\email{ leonardo.cabrer@logic.at }

\address{Luca Spada:
Dipartimento di Matematica, Universit\`a degli Studi di Salerno, Via Giovanni Paolo II 132, 84084 Fisciano (SA), Italy. 
}
\email{ lspada@unisa.it }


\begin{abstract}
We connect the dual adjunction between MV-algebras and Tychonoff spaces with the general theory of natural dualities, and provide a number of applications.  In doing so, we simplify the aforementioned construction by observing that there is no need of using \emph{presentations} of MV-algebras in order to obtain the adjunction. We also provide a description of the dual maps that is intrinsically geometric, and thus avoids the syntactic notion of \emph{definable map}. Finally, we apply these results to better explain the relation between semisimple  tensor products and coproducts of MV-algebras, and we extend beyond the finitely generated case the characterisations of strongly semisimple and polyhedral MV-algebras.
\end{abstract}

\keywords{ MV-algebras \and Adjunction \and Natural Duality \and Semisimple \and \Z-map}
\thanks{2010  {\it Mathematics Subject Classification.} 06D35 \and 03B50\and 55U10 \and 08C20}
\thanks{This research was supported by  
 Marie Curie Intra European Fellowships 
 number 299401-FP7-PEOPLE-2011-IEF and 299071-FP7-PEOPLE-2011-IEF.
 The first-named author acknowledges partial support from the Austrian Science Fund (FWF), START project Y544.
The second-named author acknowledges partial support from the Italian National Research Project (PRIN2010--11) entitled \emph{Metodi logici per il trattamento dell'informazione}.}

\maketitle

%%%%%%%%%%%%%%%%%%%%%%%%%%%%%%%%%%%%%%
%\newpage
%\tableofcontents

%%%%%%%%%%%%%%%%%%%%%%%%%%%
\section{Introduction}
\label{sect:intro}
%%%%%%%%%%%%%%%%%%%%%%%%%%%
MV-algebras were introduced in 1959 by C.\ C.\ Chang as the equivalent algebraic semantics of {\L}ukasieiwcz logic.  Since then they increasingly attracted the attention of researchers for their surprising connections with other fields of mathematics such as lattice ordered abelian groups, measure theory, quantum mechanics, toric varieties, etc. The reader is referred to \cite{Mun2011} for an updated account of their theory. 

In the last decades  a number of different approaches to describe dual categories of MV-algebras have been proposed. One of the first attempts can be found in \cite{martinez1990priestley,martinez1996simplified}, where the author provides a dual equivalence based % by ``piggy-baking'' 
on Priestley's duality for distributive lattices.  A similar approach can also be found in \cite{cabrer2006priestley}. 
 A different approach, is presented in \cite{niederkorn2001natural}; there the author directly develops dualities for finitely generated subvarieties of MV-algebras (i.e., the ones generated by a finite number of finite algebras) using the general theory of natural dualities (see \cite{NatDual88}). 

 In a different vein, in \cite{cignoli2004extending} the authors start from the duality between finite multisets and finite MV-algebras and build, using the ind and pro completions, a duality between locally finite MV-algebras and a category of Stone spaces endowed with a multiplicity function. Yet another method used to  explore categories that are dual to MV-algebras is provided by sheaf theory. Indeed, in  \cite{filipoiu1995compact} the authors prove that the whole category of MV-algebras is equivalent to the category of sheaves on a compact Hausdorff space whose stalks are local MV-algebras (i.e.,\ MV-algebras with a unique maximal ideal); although independent, this result can be seen as the translation in the setting of MV-algebras of the sheaf representation of lattice ordered abelian groups provided in \cite{MR0422107}.  Another duality, presented in \cite{DuPo2010}, connects the whole category of MV-algebras with the category of sheaves whose stalks are linearly ordered and whose base space is the spectrum of prime ideals of an MV-algebra; also in this case the duality can be seen as a consequence of an older result \cite{cornish1977chinese}.
Finally, we mention \cite{gehrke2014sheaf} and references therein for an another approach to dualities for MV-algebras based on Priestley duality, as well as for a unified treatment of sheaf representations of MV-algebras.

In this paper we will concentrate on a duality for semisimple MV-algebras presented in \cite{MarSpa2013}.  There the authors generalise the classical dual adjunction between ideals of polynomials and affine varieties to the setting of MV-algebras, thus providing an adjunction between the whole category of MV-algebras and the category of Tychonoff spaces with appropriates arrows, called ``definable maps''.  The duality for semisimple MV-algebras, as well as the one between finitely presented MV-algebras and rational polyhedra is then obtained by specialising such an adjunction.  A similar approach turns out to be possible for any variety of algebras, the reader interested on this aspect is referred to \cite{Caramello:2014aa}.

Here we will be concerned with the connection between the above geometrical dualities with the general theory of natural dualities and the applications of this new approach.

The rest of the paper is arranged in three sections. In section \ref{Sec:Preliminaries} we provide some basic definitions in universal algebra and piecewise linear geometry needed to recall the adjunction of \cite{MarSpa2013}.  
Section \ref{sec:adjoint} is devoted to showing the connections between the approach of \cite{MarSpa2013} and the theory of natural dualities.  Specifically, Theorem \ref{thm:def-maps} provides a purely geometrical characterisation of the definable maps of \cite{MarSpa2013} and Theorem \ref{Theo:DualNat} shows that the functors of \cite{MarSpa2013} are naturally equivalent to the functors arising form the theory of natural dualities.
In section \ref{sect:applications} we apply the theory developed by extending some results in MV-algebras beyond the finitely generated case. More specifically, in Theorem \ref{theo:tensor} we give a characterisation of the dual space of the (semisimple) tensor product of MV-algebras. In Theorem \ref{theorem:aereo} we give a characterisation of strongly semisimple MV-algebras that extends \cite{BuMu201X} and \cite{Cab201X}.  Finally, in Theorem \ref{Theo:CharPol} we give a characterisation of polyhedral MV-algebras that extends the results of \cite{BusCabMun201X}. 

%%%%%%%%%%%%%%%%%%%%%%%%%%%
\section{Preliminaries}
\label{Sec:Preliminaries}
%%%%%%%%%%%%%%%%%%%%%%%%%%%

%%%%%%%%%%%%%%%%%%%%%%%%%%%
\subsection{Universal algebra} 
%%%%%%%%%%%%%%%%%%%%%%%%%%%

The reader is referred to \cite{McL1998} for background in category theory and to \cite{BS1981} for universal algebra.
For convenience, let us assume in the subsequent discussion that $\lang$ is a fixed algebraic language; we call $\lang$-algebra any algebraic structure for this language.  A \word{prevariety} $\K$ is a class of algebras sharing the same language $\lang$,
that is closed under isomorphic copies, subalgebras and products; in symbols $\K=\mathbb{ISP}(\K)$. If  in addition a prevariety is closed under homomorphic images then it is called \emph{variety}. For the rest of this section $\K$ will denote a prevariety.

Given a set $X$, we indicate by $\F_{\K}(X)$ the \word{algebra freely generated by $X$ in~$\K$}, i.e., the unique algebra (up to isomorphism) for which any function $f \colon X \to B$, where $\B \in \K$,  extends uniquely to a homomorphism
$g \colon \F_{\K}(X) \to \B$.  We call the elements of $X$, the \word{free generators} of $\F_{\K}(X)$.
Let $S\seq(\F_{\K}(X))^2$ and $\theta(S)$ be the congruence of $\F_{\K}(X)$ generated by $S$. If $\A\in \K$ is  isomorphic to $\F_{\K}(X)/\theta(S)$, we say that $(X,S)$ is a \word{presentation} of $\A$. 

An algebra $\A\in \K$ is said to be a \word{finitely generated} if there exists a presentation $(\pX,S)$ where $X$ is finite; if in addition $S$ is also finite, then $\A$ is called \word{finitely presented}.  If $\K$ is a variety, for each set $X$ and $S\seq(\F_{\K}(X))^2$, the pair $(\pX,S)$ is a presentation of the algebra $\F_{\K}(X)/\theta(S)\in\K$. 

Following  \cite{MarSpa2013} we call $\K_{\mathtt{p}}$ the category whose objects are presentations of algebras in $\K$ and morphisms are homomorphisms between the corresponding quotient algebras, that is 
\begin{align*}
&h\in\Hom_{\K_{\mathtt{p}}}\big(\,(\pX,S),(\pY,T)\,\big)\quad\text{ iff }\quad h\in \Hom_{\K}\big(\,\F_{\K}(X)/\theta(S),\F_{\K}(Y)/\theta(T)\,\big).
\end{align*}
To fix the notation, we set
\begin{definition}[The functor $\Q$]\label{d:functor-Q}
The assignment $\Q_{\K}$, from $\K_{\mathtt{p}}$ to $\K$, is defined as
\begin{align*}
&\text{on objects:} & \Q_{\K}\left((\pX,S)\right):=\F_{\K}(X)/\theta(S)\\
&\text{on morphisms:}  & \Q_{\K}(h):=h
\end{align*}
\end{definition}
It is trivial to see that $\Q$ is a functor. 

Since every algebra in a (pre)variety is a quotient of a free algebra, that is, it admits a presentation, the functor $\Q_{\K}$ is dense. Since morphisms in $\K$ and $\K_{p}$ are the same and $\Q_{\K}$ acts as the identity on morphisms, $\Q_{\K}$  is full and faithful. Therefore, the categories $\K_{\mathtt{p}}$ and $\K$ are equivalent (see \cite[Theorem VI.4.1]{McL1998}).
\begin{definition}[Choice functors]\label{d:choice-functor}
Any functor $\FC\colon\K\to \K_{\mathtt{p}}$ such that $\Q_{\K}\circ \FC$ is naturally isomorphic to the identity functor on $\K$ is called a \emph{choice functor}. 
\end{definition}
The reason for the previous definition, is that each such a functor assigns an arbitrary presentation for each algebra in $\K$, in other words it ``chooses'' a presentation. Trivially, the choice functors for $\K$ are naturally equivalent.

For an algebra $\A\in\K$ and $X\seq A$, let $\st_{X}\colon \F_{\K}(X)\to \A$ be the unique homomorphism from $\F_{\K}(X)$ to $\A$ that extends the inclusion of $X$ into $A$. When $X=A$, we call $\st_{X}$ the \word{structure of $\A$} and we set $\theta_{\A}:=\ker(\st_\A)$.  Among all possible choice functors for a variety $\K$, a canonical one is provided by the \emph{structure choice functor} $\FST_{\K}\colon \K\to \K_{\mathtt{p}}$ that we describe below. 
\begin{definition}[Structure Choice Functor]\label{d:structure-functor}
Let $\iota_{\A}\colon \A\to\F_{\K}(A){/}\theta_\A$, be the isomorphism defined by $\iota_{\A}(a)=[a]_{\theta_{\A}}$, i.e., $\iota_{\A}$ maps $a$ into the equivalence class of $a$ modulo the congruence $\theta_{A}$, (see Fig.~\ref{Fig:Sta}). 
\begin{figure}[h!]
\centering
\begin{tikzpicture} 
[auto,
 text depth=0.25ex,
 move up/.style=   {transform canvas={yshift=1.9pt}},
 move down/.style= {transform canvas={yshift=-1.9pt}},
 move left/.style= {transform canvas={xshift=-2.5pt}},
 move right/.style={transform canvas={xshift=2.5pt}}] 
\matrix[row sep= 1cm, column sep= 1.3cm]  
{ 
 \node (FA) {$\F_{\K}(A)$}; &  \node (AA) {$\A$};\\
 & \node(Q) {$\F_{\K}(A)/\theta_{\A}$};  \\ 
};
\draw  [->] (FA) to node [swap]  {$[\_]_{\theta_{\A}}$}(Q);
\draw [->] (AA) to node  {$\iota_{\A}$}(Q);
\draw [->] (FA) to node  {$\st_{\A}$}(AA);
\end{tikzpicture}
\caption{}\label{Fig:Sta}
\end{figure}

Then  $\FST_{\K}\colon \K\to \K_{\mathtt{p}}$ is defined:
\begin{align*}
&\text{on objects:} & \FST_{\K}(\A):=(A,\theta_{\A}),
\end{align*}
and if $h\colon \A\to \B$  is a $\K$-homomorphism,
\begin{align*}
&\text{on morphisms:}  & \FST_{\K}(h):= \iota_{\B}\circ h\circ\iota_{\A}^{-1}
\end{align*}
It is easy to see that $\FST_{\K}$ is a well-defined functor and $\Q_{\K}\circ \FST_{\K}$ is naturally isomorphic to the identity functor on $\K$.
\end{definition}
Since in the rest of the paper we will only work with MV-algebras, we will simply write $\FST$ for $\FST_{\MV}$.

%%%%%%%%%%%%%%%%%%%%%%%%%%%
\subsection{MV-algebras}
%%%%%%%%%%%%%%%%%%%%%%%%%%%

The standard references for MV-algebras are \cite{CigDotMun2000,Mun2011}.  An MV-algebra is an abelian monoid $(A,\oplus,0)$ equipped with an operation $\neg$  such that  $\neg\neg x=x$, $x\oplus \neg 0 = \neg 0$  and $x \oplus \neg(x \oplus \neg y)=y \oplus \neg(y \oplus \neg x)$. The \word{standard MV-algebra} $[0,1]$ is the MV-algebra whose universe is the interval $[0,1]$, its operations are defined by  $x\oplus y=
\min\{x + y, 1\}$, and $\neg x = 1 - x$ and the neutral element is~$0$.

A congruence on an MV-algebra $\A$ is \word{maximal} if it is maximal in the poset of proper congruences of $\A$. 
An MV-algebra is said to be 
\word{simple} if the identity congruence is maximal, and \word{semisimple} if it is a subdirect product of simple MV-algebras.
By \cite[Theorem 3.5.1]{CigDotMun2000}, an MV-algebra is simple if and only if it is isomorphic to a subalgebra of the standard MV-algebra $[0,1]$, that is, if it belongs to $\mathbb{IS}([0,1])$. Therefore, the class $\MVS$ of semisimple MV-algebras coincides with $\mathbb{ISP}([0,1])$.
By Chang's completeness theorem (see e.g.\ \cite[Theorem 2.5.3]{CigDotMun2000}) the class of all MV-algebras coincides with $\mathbb{HSP}([0,1])$. Since $[0,1]$ generates the variety of MV-algebras, by a general result of Birkhoff, the free MV algebra on $\kappa$ generators is isomorphic to the subalgebra of $[0,1]^{[0,1]^{\kappa}}$ generated by the projections.  McNaughton in \cite{mcnaughton1951theorem}, provided a nifty characterisation of those functions.  We postpone the statement of his theorem to Theorem \ref{t:McNaughton}, as we first  need to recall some basic preliminaries in polyhedral topology.

%%%%%%%%%%%%%%%%%%%%%%%%%%%
\subsection{Polyhedral topology} 

%%%%%%%%%%%%%%%%%%%%%%%%%%%

We refer to \cite{Ew1996} for background on polyhedral topology.
A simplex $S ={\rm conv}(v_{0},\dots,v_m) \seq
\mathbb R^{n}$ is called $m$-simplex if $\{v_{0},\dots,v_m\}$ are affine linearly independent. The simplex $S$  is said to be \word{rational} if the coordinates of $v_i$ are rational numbers for each $i=0,\ldots,m$. 
A \word{polyhedron} $P$  in $\mathbb R^{n}$  is a finite union of (always closed) simplices $P=S_{1}\cup\cdots\cup S_{t}$ in $\mathbb R^n$. The set $P$ is said to be a \word{rational polyhedron} if there are rational simplices $T_{1},\ldots,T_{l}$ such that $P=T_{1}\cup\cdots\cup T_{l}$.  If $I'\seq I$ are arbitrary sets, we denote by $\pi^{I}_{I'}$ the projection map from $[0,1]^{I}$ into $[0,1]^{I'}$.

\begin{definition}[Piecewise Linear Map]\label{d:piecewise}
Let $P\seq [0,1]^{n}$ and $f\colon P\to[0,1]^{}$ be a continuous map, we say that $f$ is \emph{piecewise linear} if there exists a finite set of (affine) linear maps $f_{1},...,f_{k}\colon[0,1]^{n}\to[0,1]$ such that for every $x\in P$ there exists $1\leq i\leq k$ and $f(x)=f_{i}(x)$. 

More generally, if $P\seq [0,1]^{I}$ where I is an arbitrary set, we say that $f$ is \emph{piecewise linear} if there exists a finite subset $I'\seq I$ and a piecewise linear map $f'\colon [0,1]^{I'}\to[0,1]$ such that for every $p\in P$, $f(p)=f'(\pi^{I}_{I'}(p))$.

\end{definition}

\begin{theorem}[McNaughton Theorem]\label{t:McNaughton}
For any cardinal $\kappa$, the free MV-algebra on $\kappa$ generators is isomorphic to the algebra of continuous maps form $[0,1]^{\kappa}$ into $[0,1]$ which are piecewise linear and have integer coefficients.
\end{theorem}

%%%%%%%%%%%%%%%%%%%%%%%%%%%
\section{The adjunction between MV-algebras and Tychonoff spaces}
\label{sec:adjoint}
%%%%%%%%%%%%%%%%%%%%%%%%%%%

We start this section by briefly recalling the basic adjunction presented in~\cite{MarSpa2013}. 
We present these results using the approach to presented algebras given in Section~\ref{Sec:Preliminaries}. 

\begin{definition}[Definable Map]

Given  $P\seq[0, 1]^{I}$ and $Q\seq [0, 1]^{J}$, a function $\eta\colon P \to Q$ is
\emph{definable} if there exist elements $t_j\in\FMV(I)$ for $j\in J$, such that 
\begin{equation}\label{eq:definable}
(\eta( p ))(j) = t_j(p)=\bar{p}(t_j)\text{ for each }p\in P\text{ and }j\in J,
\end{equation}
where $\bar{p}\colon \FMV(I)\to [0,1]$ is the unique MV-homomorphism that extends the map $p\colon I\to [0,1]$.
\end{definition}

Let $\TSO$ denote the category of subspaces of the Tychonoff cubes $[0,1]^{I}$, with $I$ ranging among sets and definable maps as morphisms.
Given a set $S\seq (\FMV(I))^{2}$, we define 
\begin{align}
\label{eq:V}\mathbb{V}(S)=\{p\in [0,1]^I\mid  \bar{p}(s)= \bar{p}(t) \mbox{ for each }(s,t)\in \F^{2}(I)\}.
\end{align}
Given a set $P\seq [0,1]^I$, 
\begin{align}
\label{eq:I}\mathbb{I}(P)=\{(s,t)\in \FMV^{2}(I)\mid  \bar{p}(s)= \bar{p}(t) \mbox{ for each }p\in P\}.
\end{align}

By \cite[Theorem 3.1]{MarSpa2013},  the pair $(\mathbb{V},\mathbb{I})$ forms a Galois connection, such that 
\begin{align*}
\mathbb{V}(\mathbb{I}(P))&={\rm cl}(P)
\end{align*} 
where ${\rm cl}$ denotes the closure operator in the Euclidean topology, and 
\begin{align*}
\mathbb{I}(\mathbb{V}(S))&=\rad(S),
\end{align*} 
where $\rad(S)$ denotes the \word{radical of $S$}, i.e., the intersection of all maximal congruences of 
$\FMV(I)$ containing $S$.

\begin{definition}[The functors $\FV$ and  $\FI$]\label{d:functors-V-and-I}
Let $\FV\colon \MVp\to \TSO$ be the functor defined by:
\begin{align*}
&\text{on objects:} & \FV( \FMV(I),S)= \mathbb{V}(S)\seq [0,1]^{I}, 
%\end{align*}
\intertext{and, if $h\colon \FMV(I)/\theta(S)\to \FMV(J)/\theta(R)$ is a homomorphism and $p\in \mathbb{V}(R)$,}
%\begin{align*}
&\text{on morphisms:}  & \bigl(\FV(h)(p)\bigr)(j)= \bar{p}(t_{i}), \text{for each }i\in I
\end{align*}
where $t_i\in \FMV(J)$ is an arbitrary element of $[h(i)]_{\theta(R)}$.

Let $\FI\colon \TSO\to \MVp$ be the functor defined by:
\begin{align*}
&\text{on objects:} & \FI( P\seq [0,1]^I)= (\FMV(I), \mathbb{I}(P)), 
%\end{align*}
\intertext{and, if $P\seq [0,1]^{I}$, $Q\seq [0,1]^{J}$, $f\colon P\to Q$ is a definable map, and $t_j\in \FMV(I)$ for each $j\in J$ are defining terms for $f$, }
%\begin{align*}
&\text{on morphisms:}  & \bigl(\FI(f))\bigr)([s((X_{j})_{j\in J})]_{\mathbb{I}(Q)}= [s((X_{j}/t_{j})_{j\in J})]_{\mathbb{I}(P)},
\end{align*}
where by $s(X/t)$ we mean, as usual, the term obtained from $s(X)$ by uniformly substituting the occurrences of $X$ for $t$.
\end{definition}
In \cite[Theorem 2.8]{MarSpa2013}, the authors prove that these functors are well-defined, and that they form a dual adjunction between presented MV-algebras and Tychonoff spaces. Furthermore, by \cite[Corollary 3.2]{MarSpa2013}, when restricted to semisimple MV-algebras and to \emph{closed} subspaces of the Tychonoff cubes, $\FV$ and $\FI$ determine a dual categorical equivalence.

%%%%%%%%%%%%%%%%%%%%%%%%%%%
\subsection{Definable maps vs.  \Z-maps}
\label{sect:maps}
%%%%%%%%%%%%%%%%%%%%%%%%%%%

The restriction of $\FV$ and $\FI$ to finitely presented MV-algebras deserves special attention. 
In \cite[section~4]{MarSpa2013} it is proved that given a finite presentation $(\pm,S)$, its dual space  $\FV(\pm,S)\seq \I^m$ is a rational polyhedron, and vice versa if $\FV(\pm,S)\seq \I^m$ is a rational polyhedron then there exists a finite $T\seq (\FMV(m))^2$ such that $\rad(S)=\rad(T)$. As a consequence, the restriction of the functor $\FV$ to finite presentations determines a categorical duality between finitely presented MV-algebras and the category of rational polyhedra with definable maps. 
McNaughton theorem (Theorem \ref{t:McNaughton}), then allows a purely geometric description of the latter category by noticing that on rational polyhedra definable maps are exactly the piecewise linear maps with integer coefficients.

In \cite{BusCabMun201X}, the authors improve the above results by showing that definable maps on compact subset of $\I^n$ again admit a geometrical description\footnote{Actually, the result in \cite{BusCabMun201X} refers to compact sets in $\mathbb{R}^n$ instead of closed subsets of $\I^n$, but the difference is immaterial as noted in \cite[Claim 3.5]{MarSpa11}.}; therefore giving a purely geometrical description of the dual of the full subcategory of \emph{finitely generated} MV-algebras.
We now give a slightly more general definition of $\Z$-maps that  affords a complete, geometrical characterisation of definable maps on any subset of $\I^{I}$.
\begin{definition}[$\Z$-map] \label{definition:zeta}
Let $I,J$ be arbitrary sets and $P\seq \I^{I}$ and $Q\seq \I^{J}$.  A map $\eta\colon P\to Q$ is called a \emph{$\Z$-map} if for every $j\in J$ there exists a piecewise linear map with integer coefficients $\xi_{j}\colon \I^{I}\rightarrow \I$, such that for any $p\in P$
\[\eta(p)=\big(\xi_{j}(p)\big)_{j\in J}\ ,\]
where (by the definition of piecewise linear map) each map $\xi_{j}$ depends only upon finitely many variables.
\end{definition}

\begin{theorem}\label{thm:def-maps}
Let $I,J$ be arbitrary sets and $P\seq \I^{I}$ and $Q\seq \I^{J}$.  For any map $\eta\colon P\to Q$, the following are equivalent 
\begin{enumerate}[(i)]
\item\label{thm:def-maps:item1} the map $\eta$ is definable;
\item\label{thm:def-maps:item2} the map $\eta$ is a $\Z$-map;
\item\label{thm:def-maps:item3} for each finite subset $J'\seq J$ there exists a finite set $I'\seq I$ and a $\Z$-map $\xi\colon \I^{I'}\to \I^{J'}$ such that $\pi^{J}_{J'}\circ\eta=\xi\circ\pi^{I}_{I'}$.
\begin{center}
\begin{tikzpicture} 
[auto,
 text depth=0.25ex,
 move up/.style=   {transform canvas={yshift=1.9pt}},
 move down/.style= {transform canvas={yshift=-1.9pt}},
 move left/.style= {transform canvas={xshift=-2.5pt}},
 move right/.style={transform canvas={xshift=2.5pt}}] 
\matrix[row sep= 1cm, column sep= 1.3cm]  
{ 
\node (X) {$P$}; & \node(Y) {$Q$};  \\ 
\node (IP) {$\I^{I'}$}; & \node(JP) {$\I^{J'}$};  \\ 
};
\draw  [->] (X) to node   {$\eta$}(Y);
\draw [->] (X) to node [swap] {$\pi^{I}_{I'}$}(IP);
\draw [->] (Y) to node  {$\pi^{J}_{J'}$}(JP);
\draw  [->] (IP) to node  [swap]  {$\xi$}(JP);
\end{tikzpicture}
\end{center}
\end{enumerate}
\end{theorem}
\begin{proof}
\eqref{thm:def-maps:item1} $\Rightarrow$ \eqref{thm:def-maps:item2}.  If $\eta$ is definable then there exists a family of defining terms $t_{j}$ such that for each $p\in P$, $(\eta( p ))(j) = \bar{p}(t_j)$ for $j\in J$.  By Theorem \ref{t:McNaughton}, each term $t_{j}$ corresponds to a piecewise linear map with integer coefficients $\tau_{j}$, such that $t_{j}({x})=\tau_{j}({x})$.  Hence $\eta$ is a $\Z$-map.\\
\eqref{thm:def-maps:item2} $\Rightarrow$ \eqref{thm:def-maps:item3}.  Suppose $\eta\colon P\seq \I^{I}\to Q\seq \I^{J}$ is a $\Z$-map, then by Definition \ref{definition:zeta}, there is a family $\{\xi_{j}\}_{j\in J}$ of piecewise maps with integer coefficients such that for any $p\in P$, $\eta(p)=\big(\xi_{j}(p)\big)_{j\in J}$.  Let $J'$ be a finite subset of $J$.  By Definition \ref{d:piecewise}, for each $j\in J'$ there exists a finite subset $I_{j}\seq I$ and a piecewise linear map $\xi'_{j}\colon [0,1]^{I_{j}}\to [0,1]$, such that $\xi_{j}=\xi'_{j}\circ \pi^{I}_{I_{j}}$.  Now let $I':=\bigcup_{j\in J'} I_{j}$ and define $\xi''_{j}=\xi'_{j}\circ\pi^{I'}_{I_{j}}$.  Then
\[\pi^{J}_{J'}\circ\eta=(\xi_{j})_{j\in J'}=(\xi'_{j}\circ\pi^{I}_{I_{j}})_{j\in J'}=(\xi'_{j}\circ\pi^{I'}_{I_{j}})_{j\in J'}\circ\pi^{I}_{I'}=(\xi''_{j})_{j\in J'}\circ\pi^{I}_{I'}.\]
\eqref{thm:def-maps:item3} $\Rightarrow$ \eqref{thm:def-maps:item1}. Let us pick $J'$ in \eqref{thm:def-maps:item3} to be a singleton set $J':=\{j'\}$, then there exists a finite $I'$ and a $\Z$-map $\xi_{j'}\colon \I^{I'}\to\I$ such that $\pi_{j'}\circ\eta=\xi\circ\pi^{I}_{I'}$.  Let $t_{j'}$ be the term associated to the map $\xi$ by Theorem \ref{t:McNaughton}.  Then the last equality is equivalent to say that for any $p=(p_{i})_{i\in I}\in P$, $\big(\eta(p)\big)_{j'}=t_{j'}\big((p_{i})_{i\in I'}\big)$. Finally, by applying the same reasoning to every $j\in J$, we obtain a family of terms $t_{j}$ that define $\eta$.
\end{proof}

\begin{remark}\label{r:inf-pol}
Notice that the equivalence to item \eqref{thm:def-maps:item3} in the above theorem has a natural algebraic interpretation. Indeed, using the duality of \cite{BusCabMun201X}, the dual of  homomorphisms between finitely generated algebras are $\Z$-maps (between subset of finite dimensional cubes). Now,  if $\A$ and $\B$ are two MV-algebras and $h\colon \A\to \B$ is a map between them, then by general universal algebra, $h$ is a homomorphism if, and only if, for each finitely generated subalgebra $\C$ of $\A$, there exists a finitely generated subalgebra $\D$ of $\B$  such that $h$ restricted to $\C$ is a homomorphism into $\D$.  
\end{remark}

\begin{remark}
In the rest of the paper, we will only work with \emph{closed} subspaces of $[0,1]^{I}$ for an arbitrary set $I$.  To fix the notation, let us indicate by $\TS$ the full subcategory of $\TSO$ whose objects are closed spaces.  It is readily seen that the functor $\FV$ actually ranges in $\TS$, so in the rest of the paper we will freely think of it as a functor from $\MVp$ into $\TS$ without further notice.
\end{remark}

%%%%%%%%%%%%%%%%%%%%%%%%%%%
\subsection{From MV-algebras to Tychonoff spaces}

%%%%%%%%%%%%%%%%%%%%%%%%%%%

The functor $\FV$, of Section~\ref{sec:adjoint} is defined on each MV-algebra $\A$ by means of a presentation $(\FMV(X),S)$. This choice is justified by the authors of \cite{MarSpa2013} by the following:
\begin{quote}
`` \ldots we know of no way of associating to an abstract MV-algebra its dual object, as constructed
in this paper, other than by arbitrarily choosing a presentation of
the algebra.''
\end{quote}

In this section we give an alternative description of the dual adjunction between MV-algebras and Tychonoff spaces that is independent of the presentation. 
Our functor is based on a contravariant representable functor (in the sense of \cite[Chapter III.2]{McL1998}) and hinges on the general theory of natural dualities (as it applies to prevarieties generated by an infinite algebra \cite{piggy,NatDualChap}). One first presents a topology $\tau$ for the generating algebra $\G$, that is compact Hausdorff and such that all internal operations are $\tau$-continuous, then one considers as dual space of an algebra $\A$ the topological space $\Hom(\A,\G)$ seen as subspace of $\G^{A}$.

Following the usual notation in natural dualities, let  $\uline{\I}$ denote the standard MV-algebra on $[0,1]$ and $\twiddle{\I}$ denote the topological subspace of the real numbers with the Euclidean topology. 

\begin{definition}[The functor $\FW$]\label{d:functor-W}
Let $\FW\colon \MV\to\TS$ be the assignment  defined by:
\begin{align*}
&\text{on objects:} & \FW( \A)= \Hom(\A,\uline{\I}),
%\end{align*}
\intertext{and, if $h\colon \A\to \B$ is an MV-homomorphism and $g\in \FW( \B)$,}
%\begin{align*}
&\text{on morphisms:}  & \bigl(\FW(h)\bigr)(g)=g \circ h,
\end{align*}
where  $\Hom(\A,\uline{\I})$  is seen as a subspace of $\twiddle{\I}^{\A}$. 
\end{definition}
\begin{remark}\label{Max-W-homeo}
Observe that there is a bijective correspondence from $\Hom(\A,\uline{\I})$ into $\MAX(\A)$ (the set of maximal ideals of $\A$): it is given by $f\mapsto \{a\in A\mid f(a)=0\}$.  If one endows $\MAX(\A)$ with the topology generated by the closed elements of the form $V_{f}=\{M\in\MAX(\A)\mid f\in M\}$, then the above correspondence turns out to be a homeomorphism. 
\end{remark}
\begin{theorem}\label{Theo:RepFunctor}
For each $\A\in \MV$ and each MV-homomorphism $h$,
\begin{enumerate}
\item\label{Theo:RepFunctor:item1} $\FW(\A)$ is a closed subset of $\twiddle{\I}^{\A}$, and
\item\label{Theo:RepFunctor:item2} $\FW(h)$ is a well-defined continuous map.
\end{enumerate}
So, $\FW$ is a well-defined functor.
\end{theorem}
\begin{proof}
By Remark \ref{Max-W-homeo}, we obtain the compactness of $\FW(\A)$ immediately from the compactness of $\MAX(\A)$ (for the compactness of $\MAX(\A)$ see e.g., \cite[Proposition 4.15]{Mun2011}). Since $[0,1]^{\A}$ is Hausdorff, \eqref{Theo:RepFunctor:item1} follows. 

To prove \eqref{Theo:RepFunctor:item2}, let $h\colon \A\to \B$ be an MV-homomorphism. The fact the $\FW(h)\colon\FW(\B)\to \FW(\A)$ is a well-defined map, follows directly from the fact that $g\circ h\colon \A\to \uline{\I}$ is an MV-homomorphism for each $g\in\Hom(\B,\uline{\I})$.

To see that it is continuous it is enough to prove that $\pi_a\circ \FW(h) $ is continuous for each $a\in \A$. It is easy to see that $\pi_a\circ \FW(h)=\pi_{h(a)}$ (see Fig.~\ref{Fig:FW}).
\begin{figure}[h!]
\centering
\begin{tikzpicture} 
[auto,
 text depth=0.25ex,
 move up/.style=   {transform canvas={yshift=1.9pt}},
 move down/.style= {transform canvas={yshift=-1.9pt}},
 move left/.style= {transform canvas={xshift=-2.5pt}},
 move right/.style={transform canvas={xshift=2.5pt}}] 
\matrix[row sep= .7cm, column sep= 1cm]  
{ 
\node (A) {$\A$};  & \node (B) {$\B$}; & & \node (VB) {$\FW(\B)$}; & & \node (VA) {$\FW(\A)$};\\
  &  & & \node (IB) {$\twiddle{\I}^{B}$}; & & \node (IA) {$\twiddle{\I}^{A}$};\\
  & \node (I) {$\uline{\I}$};& &  & \node (I2) {$\twiddle{\I}$};& \\
  };
\draw  [->] (A) to node  {$h$}(B);
\draw [->] (A) to node [swap]  {$\FW(h)(g)$}(I);
\draw [->] (B) to node  {$g$}(I);
\draw  [->] (VB) to node  {$\FW(h)$}(VA);
\draw  [right hook->] (VB) to node [swap]  {}(IB);
\draw [right hook->] (VA) to node  {}(IA);
\draw [->] (IA) to node [swap]  {$\pi_a$}(I2);
\draw [->] (IB) to node  {$\pi_{h(a)}$}(I2);
\end{tikzpicture}
\caption{}\label{Fig:FW}
\end{figure}
Indeed,  for each $g\in \FW(\B)$, 
\[(\pi_a\circ \FW(h))(g)=\pi_a(g\circ h)=(g\circ h)(a)=g(h(a))=\pi_{h(a)}(g).\]
Then $\pi_a\circ \FW(h) $ is continuous for each $a\in \A$, and so is $\FW(h)$.
\end{proof}
\begin{remark}
Unfortunately, in item \eqref{Theo:RepFunctor:item1} of the above proof, the role played by the continuity of the operations in $\uline{\I}$ is hidden. However, in the setting of natural dualities, this property is central to establish that the functors considered are well-definite (see \cite[Exercise 2.9]{NatDual88}). Therefore, in the spirit of this section, we show how to prove that $\Hom(\A,\uline{\I})$  is a closed subset of $\twiddle{\I}^{\A}$ using the general argument of natural dualities.

We show that $\Hom(\A,\uline{\I})$ contains all its limit points.  To this end, pick any $f\not\in \Hom(\A,\uline{\I})$. Trivially, one of the following conditions must hold:
\begin{enumerate}[(i)]
\item\label{item:1} $f(0)\neq 0$.
\item\label{item:2} There exists $a\in A$, with $f(\neg a)\neq \neg f(a)$.
\item\label{item:3} There exist $a,b\in A$, with $f(a\oplus b)\neq f(a)\oplus f(b)$.
\end{enumerate}
Assume \eqref{item:3} is the case. Since $\twiddle{\I}$ is Hausdorff, there exist open sets $U,V\seq \I$ such that $f(a\oplus b)\in U$, $ f(a)\oplus f(b)\in V$ and $U\cap V=\emptyset$. Since $\oplus\colon \I^2\to\I$ is continuous the inverse image of $V$ is open in $[0,1]^{2}$.  In other words there exist open sets  $V_a,V_b\seq \I$ such that
\begin{align}
\label{Eq:OpenOplus}x\oplus y\in V\mbox{ if, and only if, }(x,y)\in V_a\times V_b.
\end{align}
Notice that, since $f(a)\oplus f(b)\in V$ one has $f(a)\in V_a$, $f(b)\in V_b$.
 
For each $c\in A$, let $\pi_{c}\colon \I^{\A}\to[0,1]$ be defined as the evaluation $\pi_{c}(h)=h(c)$, for $h\in [0,1]^{A}$. Consider also the set
\[O=\pi_a^{-1}(V_a)\cap \pi_b^{-1}(V_b)\cap \pi_{a\oplus b}^{-1}(U),\]
which is open because the evaluations are continuous maps.

From $f(a)\in V_a$, $f(b)\in V_b$ and $f(a\oplus b)\in U$, one immediately has $f\in O$. On the other hand,  for any $g\in O$, one has that $g(a\oplus b)\in U$ and, by \eqref{Eq:OpenOplus}, that $g(a)\oplus g(b)\in V$. Since  $U\cap V=\emptyset$, $g(a\oplus b)$ must differ from $g(a)\oplus g(b)$, hence, $g\notin  \Hom(\A,\uline{\I})$. Thus, $O$ and  $\Hom(\A,\uline{\I})$ are disjoint.  Therefore, $f$ is not a limit point of $\Hom(\A,\uline{\I})\subseteq \twiddle{\I}^{A}$ and the claim in proved. The cases \eqref{item:1} and \eqref{item:2} can be dealt using similar arguments. 
\end{remark}

\begin{theorem}\label{Theo:DualNat}
Let the functors $\FST$, $\FV$ and $\FW$ be defined as in Definitions \ref{d:structure-functor}, \ref{d:functors-V-and-I}, and \ref{d:functor-W}, respectively.  Then $\FW=\FV\circ\FST$, so for each choice functor $\FC\colon\MV\to\MVp$,  the functor $\FW$ is naturally equivalent to the functor $\FV\circ \FC$.
\end{theorem}
\begin{proof}

We start by proving that $\FW=\FV\circ \FST$. Given $p\in \I^{A}$, recall that $\bar{p}\colon \F(A)\to\uline{\I}$ is the unique MV-homomorphism that extends $p$. It follows form general algebra that a homomorphism from $h\colon \F(A)\to\uline{\I}$ factors through $\st_{\A}$ if and only if $\ker(h)\supseteq \ker(\st_{\A})=\theta_{\A}$ (see Fig.\ \ref{Fig}).

\begin{figure}[h!]
\centering
\begin{tikzpicture} 
[auto,
 text depth=0.25ex,
 move up/.style=   {transform canvas={yshift=1.9pt}},
 move down/.style= {transform canvas={yshift=-1.9pt}},
 move left/.style= {transform canvas={xshift=-2.5pt}},
 move right/.style={transform canvas={xshift=2.5pt}}] 
\matrix[row sep= 1cm, column sep= 1.3cm]  
{ 
\node (AS) {$\F(A)$}; & \node (I) {$\uline{\I}$}; & & \node (FA2) {$\F(A)$}; &  \node (I2) {$\uline{\I}$}; \\
   & \node (FA) {$A$}; & & \node(Q) {$\F(A)/\theta_{\A}$}; & \node (AA) {$\A$}; \\ 
};
\draw  [->] (AS) to node  {$\bar{p}$}(I);
\draw [->] (FA) to node [swap]  {${p}$}(I);
\draw [->] (AS) to node [below] {$\st_{\A}$} (FA);
\draw  [->] (FA2) to node  {$h$}(I2);
\draw  [->] (FA2) to node [swap]  {$[\_]_{\theta_{\A}}$}(Q);
\draw [->] (Q) to node [swap]  {$\iota_{\A}^{-1}$}(AA);
\draw [->] (AA) to node [swap]  {$h'$}(I2);
\draw [->] (FA2) to node  {$\st_{\A}$}(AA);
\end{tikzpicture}
\caption{}\label{Fig}
\end{figure}

 Thus, for each MV-algebra $\A$
\begin{equation}
p\in \Hom(\A,\uline{\I}) \mbox{ iff }\ker(\bar{p})\supseteq\theta_{\A}.
\end{equation}
Therefore,
\begin{align*}
\FV\circ \FST(\A)&=\mathbb{V}(\theta_{\A})=\{p\in\I^{A}\mid s(p)=t(p), \mbox{ for each }(s,t)\in\theta_{\A} \}\\
&=\{p\in\I^{A}\mid \bar{p}(s)=\bar{p}(t), \mbox{ for each }(s,t)\in\theta_{\A}\}\\
&=\{p\in\I^{A}\mid \ker(\bar{p})\supseteq \theta_{\A}\}\\
&=\Hom(\A,\uline{\I}).
\end{align*}

Recall that $\FV$ sends a homomorphism $f\colon \F(I)/\theta\to\F(J)/\theta'$ into the function 
\begin{align}
\label{eq:func-V}\FV(f)(p)=(f_x(p))_{x\in \mu}=(\bar{p}(f_i))_{i\in I}\text{ for each }p\in \I^{J},
\end{align}
where $f_{i}\in f([i]_{\theta})$ is arbitrarily chosen for each $i\in I$.
Now let $h\colon \A\to \B$ be an MV-homomorphism. For each $a\in A$, 
\begin{align}
\label{eq:func-S}\FST(h)([a]_{\theta_{\A}})=\iota_{\B}\circ h\circ\iota_{\A}^{-1}([a]_{\theta_{\A}})=\iota_{\B}(h(a))=[h(a)]_{\theta_{\B}}.
\end{align}
Then, fixing an arbitrary $f_a\in[h(a)]_{\theta_{B}}$ for each $a\in A$, it follows from \eqref{eq:func-V} and \eqref{eq:func-S}
\begin{align*}
((\FV\circ\FST)(h)) (p)&=\left(\FV(\FST(h))\right) (p)=(\bar{p}(f_a))_{a\in A}=(\bar{p}(h(a)))_{a\in A}\\
&= (p(h(a))_{a\in A}
=((p\circ h)(a)))_{a\in A}= p\circ h,
\end{align*}
for each $p\in \mathbb{V}(\theta_{\B})=\Hom(\B,\uline{\I})$. Which finishes the proof that $\FV\circ \FST=\FW$.  As an immediate consequence, we obtain that $\FW$ is naturally isomorphic to $\FV\circ \FC$ for any  choice functor $\FC\colon \MV\to\MVp$. 
\end{proof}

\begin{remark}
There is an advantage in working with presentations and the operator $\mathbb{V}$ rather than $\Hom$: the choice of a presentation allows more control on the dimension of the Tychonoff spaces associated with the algebras. E.g., when a semisimple MV-algebra $\A$ is generated by a finite set of elements $a_1,\ldots, a_m\in A$ then $\A\cong \F(m)/\theta$ for some $\theta$ congruence of $\F(m)$. Using  $\theta$ instead of $\theta_{\A}$ we have $\FV(\A)=\mathbb{V}(\theta)\seq \twiddle{[0,1]}^{m}$ as opposed to $\FW(\A)=\mathbb{V}(\theta_{\A})\seq \twiddle{[0,1]}^{A}$.
\end{remark}

\begin{remark}\label{Rem}
 It is easy and useful to explicitly see how the natural isomorphism from $\FV\circ\FC$ to $\FW$ works for any choice functor.  
Let us fix a choice functor $\FC\colon\MV\to\MVp$.  By Definition \ref{d:choice-functor}, the pair $\FC,\Q$ is an equivalence of categories, thus let $\rho\colon {\rm Id}_{\MV}\to \Q\circ\FC$ be its unit. Recalling (Definition \ref{d:structure-functor}) that $\iota_{\A}$ are the components of a natural isomorphism from ${\rm Id}_{\MV}$ to $\Q\circ\FST$. Then the assignment $\A\mapsto \FV(\FC(\iota_{\A}\circ\rho_{\A}^{-1})$ is the natural isomorphism from   $\FV\circ\FC$ to $\FW$. (See Fig.~\ref{Fig:NatIso}.)
\begin{figure}[h!]
\centering
\begin{tikzpicture} 
[auto,
 text depth=0.25ex,
 move up/.style=   {transform canvas={yshift=1.9pt}},
 move down/.style= {transform canvas={yshift=-1.9pt}},
 move left/.style= {transform canvas={xshift=-2.5pt}},
 move right/.style={transform canvas={xshift=2.5pt}}] 
\matrix[row sep= 1cm, column sep= .8cm]  
{ 
\node (FUT) {$\F(U)/\theta$}; & \node (FCA) {$\FC(\A)=(\F(U),\theta)$}; &  \node (VU) {$\FV(\FC(\A))=\mathbb{V}(\theta)$};\\
\node (A) {$\A$}; & \\
\node(FAT) {$\F(A)/\theta_{\A}$}; & \node (FSTA) {$\FST(\A)=(\F(A),\theta_{\A})$}; & \node (VA) {$\FV(\FST(\A))=\mathbb{V}(\theta_{\A})\cong\Hom(\A,\uline{\I})=\FW(\A)$};\\
};
\draw [->] (FUT) to node  {$\rho_{\A}^{-1}$}(A);
\draw  [->] (A) to node  {$\iota_{\A}$}(FAT);
\draw [<-] (FCA) to node [right]  {$\FC(\iota_{\A}\circ\rho_{\A}^{-1})$}(FSTA);
\draw [<-] (VU) to node {$\FV(\FC(\iota_{\A}\circ\rho_{\A}^{-1}))$}(VA);
\end{tikzpicture}
\caption{}\label{Fig:NatIso}
\end{figure}

We now prove a simple but useful lemma that describes how $\FW$ operates on subalgebras generated by a finite set of generators.
\end{remark}
\begin{lemma}\label{Lem:useful}
Let $\A$ be a semisimple MV-algebra, let $C\seq A$ and let $\B$ the subalgebra of $\A$ generated by $C$. The dual of $\B$, $\FW(\B)$ is $\Z$-homeomorphic to $\pi^{A}_{C}(\FW(\A))$.
\end{lemma}
\begin{proof}
Since $\B$ is generated by $C$, there exists a choice functor $\FC$, such that  $\FC(\B)=(\FMV(C),\theta)$. It is easy to see that the map  $\rho_{\B}(c)=[c]_{\theta}$ for each element $c\in C$ extends to an isomorphism from $\B$ into $\FMV(C)/\theta$. Let ${\rm inc}\colon \B\to \A$ be the inclusion homomorphism, and set  $f= \iota_{\A}\circ{\rm inc}\circ \rho_{\B}^{-1}$.  Thus, $f\colon \FMV(C)/\theta\to\FMV(A)/\theta_{\A}$ is such that $f([c]_{\theta})=[c]_{\theta_{\A}}$. 
Now, to apply the functor $\FV$ of Definition \ref{d:functors-V-and-I} to $f$, we need to select an arbitrary element of $[c]_{\theta_{\A}}$, for each $c\in C$.  If we pick $c\in [c]_{\theta_{\A}}$ we can write
\begin{equation}\label{eq:proj}
\FV(f)\colon \mathbb{V}(\theta_{A})\to \mathbb{V}(\theta),\qquad \FV(f)(p)=(\bar{p}(c))_{c\in C}=\pi^{A}_C(p),
\end{equation}
for each $p\in \mathbb{V}(\theta_{\A})$.
Finally, by Theorem \ref{Theo:DualNat}, $\FW(\B)$ is $\Z$-homeomorphic to $\FV(\FC(\B))$ so we conclude that $\FW(\B)$ is $\Z$-homeomorphic to $\pi^{A}_{C}(\FW(\A))$.
\end{proof}

%%%%%%%%%%%%%%%%%%%%%%%%%%%
\subsection{From Tychonoff spaces to MV-algebras}

%%%%%%%%%%%%%%%%%%%%%%%%%%%

Extending the parallel between this representation and natural dualities we set
\begin{definition}[The functor $\McN$]\label{d:functor-M}
Let $\McN\colon \TS\to\MV$ be the assignment  defined by:
\begin{align*}
&\text{on objects:} & \McN( X)= \TS(X,{\twiddle{\I}}), \text{ for each }X\in \TS, 
%\end{align*}
\intertext{and if $\eta\in\TS( X, Y)\text{ and }\xi\in \McN(Y),$}
%\begin{align*}
&\text{on morphisms:} & \bigl(\TS(\eta)\bigr)(\xi)=\xi \circ \eta,
\end{align*}
where  $\TS(X,\twiddle{\I})$  is seen as a subalgebra of  $\uline{\I}^{X}$. 
\end{definition}
\begin{proposition}\label{p:adjunction}
The assignment $\McN$ is a well-defined contravariant functor. Moreover, for each MV-algebra $\A$, the map $\ev_{\A}\colon  \A\to \McN(\FW(\A))$ defined by $(\ev_{\A}(a))(f)=f(a)$ is an onto map and the assignment $\A\mapsto \ev_{\A}$ is a natural transformation from ${\rm Id}_{\MV}$ to $\McN\circ\FW$.
Similarly, $\cev_{X}\colon X\to \FW(\McN(X))$ defined by  $(\cev_{\A}(x))(g)=g(x)$ determines a natural transformation from  ${\rm Id}_{\TS}$ to $\FW\circ \McN$.
\end{proposition}
\begin{proof}
For any object $X$ in $\TS$ the image $\McN(X)$ is an MV-algebra, for by Theorem \ref{thm:def-maps}, $\Z$-maps coincide with definable maps, and it is readily seen that definable maps are closed under MV-operations.
Further, it is straightforward to check that $\McN(\eta)$ is an MV-homomorphism for each $\eta\in\TS( X, Y)$, hence $\McN$ is a well-defined functor. 

For $\A\in \MV$ we first prove that $\ev_{\A}$ is well defined. Indeed, if $a\in \A $, since $\FW(\A)\seq [0,1]^{A}$, we have that $\ev_{\A}(a)$ coincides with the projection $\pi^{A}_{a}$, that is trivially a $\Z$-map.  The fact that it is an MV-homomorphism follows directly from the definition of $\ev_{A}$ and the fact that elements of $\FW(\A)$ are MV-homomorphisms. 

To see that the association $\A\mapsto\ev_{\A}$ is natural in $\A$ consider an MV-homomorphism $h\colon\A\to\B$ and let $f\in \FW(\B)$.  Then 
\begin{align*}
\ev_{\B}(h(a))(f)&=f(h(a))=f\circ h (a)\\
&= \ev_{\A}(a)(f\circ h) \\
&=(\ev_{\A}(a)\circ \FW(h))(f)\\
&=\Big(\McN\FW(h)(\ev_{\A}(a))\Big)(f)
\end{align*}
for each $a\in\A$.

To prove that $\ev_{\A}$ is onto let $\eta\in\McN\FW(A)= \TS(\FW(\A),\twiddle{\I})$. 
Since by Theorem~\ref{thm:def-maps}, $\eta$ is a definable map, there exists a term $t\in\FMV{(A)}$ such that for any $p\in\FW{(A)}$,  $\eta(p)=t(p)$.
Let us fix  $a=\st_{\A}(t)$. We claim that $\eta=\ev_{\A}(a)$. Indeed, $p\in\FW{(A)}$,
\begin{align*}
\eta(p)&=t(p)=p([t]_{\theta_{\A}})=p\circ \st_{\A}(t)=f(a).
\end{align*}
The fact the $\cev$ is a natural transformation from  ${\rm Id}_{\TS}$ to $\FW\circ \McN$ follows by similar arguments.
\end{proof}
\begin{remark}
The natural duality expert will recognise in Proposition~\ref{p:adjunction} the adjunction between the natural functors determined by $\uline{\I}$ and its alter ego $\twiddle{\I}$. The fact that the representable functors are well defined and that the evaluation maps determine a dual adjunction is ensured by the compatibility conditions between the algebra and its alter ego. In our case, since the alter ego is just a totopological space, without relational of functional structure, the compatibility condition reduces to the continuity of the operations in $\uline{\I}$ with respect to the topology of $\twiddle{\I}$. We included a short proof here for completeness. 
\end{remark}

\begin{remark}
It is worthwhile to notice explicitly that the functors $\McN$ and $\FI$ are naturally isomorphic, i.e., for any set $I$ and any subset of $X\seq [0,1]^{I}$ there is an isomorphism between $\FI(X)$ and $\McN(X)$, natural in $X$.  Indeed, recall that $\FI(X)$ corresponds to the algebra $\F(I)/\mathbb{I}(X)$ (with the notation in the equation \eqref{eq:I}), while by Theorem \ref{thm:def-maps}, $\McN(X)$ is the algebra of definable functions restricted to $X$.  The isomorphism is hence described by sending a generic element $t(x)/\mathbb{I}(X)$ into the definable function given by $t$.  It is straightforward to see that this is a well-defined, bijective homomorphism. By way of example we show injectivity: suppose that $s$ and $t$ are two terms whose corresponding definable functions are equal on $X$, then by definition (see equation \eqref{eq:I}) the pair $(s,t)$ belongs to $\mathbb{I}(X)$, so $s/\mathbb{I}(X)=t/\mathbb{I}(X)$ in $\F(I)/\mathbb{I}(X)$ and the claim is proved.
\end{remark}

\begin{theorem}
The functor $\McN$ and the restriction of $\FW$ to $\MVS$ determine a dual categorical equivalence between $\TS$ and $\MVS$.
\end{theorem}
\begin{proof}
Observe that by \cite[Theorem 4.16]{Mun2011}, for each $X\in \TS$, the MV-algebra $\McN(X)$ is a separating subalgebra of $\uline{\I}^X$, hence it is semisimple. So, in fact the functor $\McN$ ranges in $\MVS$.

We have already observed that $\ev_{\A}\colon \A\to \McN(X)$ is an onto map for each $\A\in \MV$. If $\A\in \MVS$, then $\ev_{\A}$ is also one-to-one this follow directly from the fact that $\A\in\mathbb{ISP}(\uline{\I})$.
Therefore, $\ev$ is a natural isomorphism between  ${\rm Id}_{\MV}$ and $\McN\circ\FW$. Similarly, using \cite[Theorem 4.16]{Mun2011}, it follows that $\cev$ is  a natural isomorphism between  ${\rm Id}_{\TS}$ and 
$\FW\circ\McN$. 
\end{proof}

\begin{remark}\label{r:embedding}
There are two differences between our presentation and the natural duality approach.  Firstly, we keep the embedding $\Hom(\A,\G)\seq \G^{\A}$ into a Tychonoff cube as an intrinsic part of the dual space. In other words, we keep track of a certain coordinate system for the topological space. The importance of this statement cannot be overestimated. The notion of $\Z$-map between closed subspaces of Tychonoff cubes relies on the enveloping cubes. Without the coordinate systems the notion of projection map does not make sense. Moreover, our definition of $\McN$ rests upon $\Z$-maps, and therefore depends on the enveloping cube. 

The second important difference is that our description of dual of morphisms is not obtained as structure preserving maps, but in a geometric way depending on how each subspace is embedded into a Tychonoff cube.  
\end{remark}

%%%%%%%%%%%%%%%%%%%%%%%%%%%
\section{Applications}
\label{sect:applications}
%%%%%%%%%%%%%%%%%%%%%%%%%%%

%%%%%%%%%%%%%%%%%%%%%%%%%%%
\subsection{Coordinate systems and tensor products}

%%%%%%%%%%%%%%%%%%%%%%%%%%%

In this subsection we describe the dual of the semisimple tensor product of semisimple MV-algebras. This provides a clear example to our earlier observation on the significance of keeping track of the embedding of the dual space into a Tychonoff cube (see Remark \ref{r:embedding}). We refer the reader to \cite[Chapter~IX]{Mun2011} for the definition and construction of MV-tensor product and the semisimple tensor product.   According to Remark \ref{Max-W-homeo}, we identify the maximal spectrum of $\A$ with $\FW(\A)$.  Let us fix for the remainder of  this section two semisimple MV-algebras $\A$ and $\B$, and let $\A\otimes \B$ and $\A\coprod \B$ denote their semisimple tensor product and coproduct, respectively.

In \cite[Example~9.15]{Mun2011} it is observed that even for finite (whence semisimple) MV-algebras the (semisimple) tensor product and the coproduct do not coincide. However, as shown below,  the maximal spectrum of the tensor product $\A\otimes \B$ is homeomorphic to the maximal spectrum of the coproduct $\A\coprod \B$.  This happens because even if $\FW(\A\otimes\B)$ is homeomorphic to $\FW(\A\coprod\B)$, their enveloping cubes are different as we show in the remainder of the section.

We start with an easy corollary of duality result obtained in the previous section.
\begin{corollary}\label{cor:W(A)xW(B)}
The maximal spectra $\FW(\A\coprod\B)$  and $\FW(\A)\times\FW(\B)$ are  $\Z$-home\-omorphic.
\end{corollary}
\begin{proposition}
The maximal spectra $\FW(\A\otimes\B)$  and $\FW(\A)\times\FW(\B)$ are homeomorphic.
\end{proposition}
\begin{proof}
In \cite[Theorem~9.17]{Mun2011} it is proved that $\A\otimes \B$ is isomorphic to the subalgebra $\C$ of continuous functions in $\twiddle{\I}^{(\FW(\A)\times\FW(\B))}$ generated by the maps $(f\cdot g)(x,y)=f(x)\cdot g(y)$ for $(f,g)\in \McN(\FW(\A))\times \McN(\FW(\B))$,  where $f(x)\cdot g(y)$ denotes the usual product of $f(x)$ and $g(y)$ in $\I$. It is also observed there that this algebra is separating.  Hence, by \cite[Theorem~3.4.3]{CigDotMun2000}, $\FW(\C)\cong  \FW(\A)\times \FW(\B)$.

\end{proof}
\begin{corollary}\label{cor:W(A)xW(B)2}
The maximal spectra $\FW(\A\coprod\B)$  and $\FW(\A\otimes\B)$ are  home\-omorphic.
\end{corollary}
Finally we prove the following theorem.
\begin{theorem}\label{theo:tensor}
Let $\cdot\colon \I^A\times \I^B\to \I^{A\times B}$
be the map defined by 
\begin{align*}
(f\cdot g)(a,b):= f(a)\cdot g(b)
\end{align*}
for each $(f,g)\in\I^A\times \I^B$.

The spaces  $\FW(\A\otimes\B)\seq\twiddle{\I}^{\A\otimes \B}$ and $\cdot(\FW(\A)\times\FW(\B))\seq \twiddle{\I}^{A\times B}$ are  $\Z$-homeomorphic.
\end{theorem}
\begin{proof}
By \cite[Construction~9.12 and Proposition~9.13]{Mun2011}, $\A\otimes\B$ admits the presentation $(A\times B,R)$ where $R$ is the smallest set containing, for any $a, a_{1},a_{2}\in A$ and any $b,b_{1},b_{2}\in B$:
\begin{enumerate}
\item $((1,1),\mathbf{1})$, $((a,0),\mathbf{0})$ and $((0,b),\mathbf{0})$
\item $\big((a,b_1\vee b_2),(a,b_1)\vee(a, b_2)\big)$
\item $\big((a_1\vee a_2,b),(a_1,b)\vee(a_2, b)\big)$
\item $\big((a,b_1\wedge b_2),(a,b_1)\wedge(a, b_2)\big)$
\item $\big((a_1\wedge a_2,b),(a_1,b)\wedge(a_2, b)\big)$
\item $\big((a,b_1)\odot(a, b_2),0\big)$, whenever $b_1\odot b_2=0$
\item $\big((a_1,b)\odot(a_2, b),0\big)$, whenever $a_1\odot a_2=0$
\item $\big((a,b_1\oplus b_2),(a,b_1)\oplus(a,b_2)\big)$, whenever $b_1\odot b_2=0$
\item $\big((a_1\oplus a_2, b),(a_1,b)\oplus(a_2,b)\big)$, whenever $a_1\odot a_2=0$.
\end{enumerate}
So, by Theorem \ref{Theo:DualNat}, $\FW(\A\otimes\B)$ is $\Z$-homeomorphic to $\FV(A\times B,R)=\mathbb{V}(R)\seq \I^{A\times B}$. 
We claim that $\mathbb{V}(R)=\cdot(\FW(\A)\times\FW(\B))$. 
First observe that for each $(p,q),(r,s)\in \FW(\A)\times\FW(\B))$,$\cdot(p,q)=\cdot(r,s)$ implies $(p,q)=(r,s)$. Indeed, since the maps $p$, $q$, $r$, and $s$ are MV-homomorphisms, $p(1)=r(1)=q(1)=s(1)=1$, so, if $\cdot(p,q)=\cdot(r,s)$ then $p=r$ and $q=s$. 

We prove that $\cdot(\FW(\A)\times\FW(\B))\seq \mathbb{V}(R)$. Let $r\in\cdot(\FW(\A)\times\FW(\B))$, so there are $p\in \FW(\A)\seq\I^{A}$ and $q\in\FW(\B)\seq \I^{B}$ such that $r=p\cdot q$. We have 
\begin{align*}
r(1,1)&=p(1)\cdot q(1)=1=r(\mathbf{1})\\
r(a,0)&=p(a)\cdot q(0)=0=r(\mathbf{0})\\
r(a,b_1\vee b_2)&=p(a)\cdot q(b_1\vee b_2)=p(a)\cdot (q(b_1)\vee q(b_2))\\
&=(p(a)\cdot q(b_1))\vee (q(a)\cdot p(b_2))\\
&=r(a,b_1)\vee r(a, b_2)=r((a,b_1)\vee(a, b_2)).
\end{align*}
If $b_1\odot b_2=0$, since $q(b_1)\odot q(b_2)=0$ and $q(b_1)\oplus q(b_2)=q(b_1)+q(b_2)$, we have
\begin{align*}
r((a,b_1)\odot(a, b_2))&=r(a,b_1)\odot r(a,b_2)=(p(a)\cdot q(b_1))\odot(p(a)\cdot q(b_2))\\
&=p(a)\cdot(q(b_1)\odot q(b_2))=0=r(0))\\
r(a,b_1\oplus b_2)&=p(a)\cdot q(b_1\oplus b_2)\\
&=p(a)\cdot (q(b_1)\oplus q( b_2))=p(a)\cdot (q(b_1)+q( b_2))\\
&=(p(a)\cdot q(b_1))+ (p(a)\cdot q( b_2))=r(a,b_1)+r(a,b_2)\\
&=r(a,b_1)\oplus r(a,b_2)\\
&=r((a,b_1)\oplus(a,b_2)).
\end{align*}
The rest of the necessary conditions to ensure that $p\cdot q$ is in $\mathbb{V}(R)$ can be checked in a similar way.

To prove that $ \mathbb{V}(R) \seq\cdot(\FW(\A)\times\FW(\B))$, let $r\in \mathbb{V}(R)$.
Observe  that the map $r\colon A\times B\to [0,1]$ is a bimorphism. Then by \cite[Lemma~9.11]{Mun2011}, the map $p\colon A\to \I$ defined by $p(a)=r(a,1)$ is an MV-homomorphism. 
Similarly $q\colon {B}\to \I$ defined by $q(b)=r(1,b)$ is an MV-homomorphism. Therefore $(p,q)\in \FW(\A)\times \FW(\B)$.
Finally by \cite[Lemma~9.16]{Mun2011},  $r(a,b)=p(a)\cdot q(b)$. 
\end{proof}

Summing up, by Corollary \ref{cor:W(A)xW(B)}, the dual of the coproduct $\FW(\A\coprod\B)$ is $\Z$-homeomorphic to $\FW(\A)\times\FW(\B)\seq \twiddle{\I}^{A\uplus B}$ (where $\uplus$ denote disjoint union) while by Theorem \ref{theo:tensor} the dual of the tensor product, $\FW(A\otimes B)$ is $\Z$-homeomorphic to $\cdot(\FW(\A)\times \FW(\B))\seq \twiddle{\I}^{A\times B}$. Therefore, although the two spaces are homeomorphic, they differ on the enveloping space.

%%%%%%%%%%%%%%%%%%%%%%%%%%%
\subsection{Strongly semisimple MV-algebras}
\label{sec:strongly}
%%%%%%%%%%%%%%%%%%%%%%%%%%%

In this subsection and the following one we apply the duality of the previous section to extend the characterisations of two kinds of MV-algebras that up to now only work for finitely generated MV-algebras.   A crucial observation will be that the properties to be checked for the characterisations depend only on finitely generated subalgebras.

Following Dubuc and Poveda \cite{DuPo2010},
we say that an MV-algebra $A$  is
\word{strongly semisimple}  if all its quotients over a principal ideal are semisimple. 
In  \cite{BuMu201X}, Busaniche and Mundici characterised $2$-generated strongly semisimple MV-algebras, and in \cite{Cab201X} the first author generalised their result to all finitely generated MV-algebras.
This characterisation is based on the $n$-dimensional generalisation of Bouligand-Severi tangents (see \cite{Bo1930,Se1927,Se1931}). 
We extend here the characterisation to all strongly semisimple MV-algebras, but first  we need to recall some definitions.

Given a $k$-tuple $u=(u_1,\ldots,u_k)$ of pairwise orthogonal unit vectors in $\mathbb R^n$, for each   $l\leq k$, let 
\[\mathsf{p}_{l}\colon \mathbb R^n\to \mathbb{R}u_1+\cdots+\mathbb{R}u_{l}\] 
denote the orthogonal projection map onto the linear subspace of $\mathbb R^n$ generated by $u_1,\ldots,u_l$.

\begin{definition}[\mbox{\cite[Definition~3.1]{Cab201X}}]
Let $X$ be a closed subset of $\twiddle{\I}^n$ and $u=(u_1,\ldots,u_k)$ be a $k$-tuple of  pairwise orthogonal unital vectors in $\mathbb R^n$.
We say that $u$ is a  \word{Bouligand-Severi tangent of $X$ at $x$ of degree $k$}, or simply a $k$-tangent of $X$ at $x$, if $X$ contains a sequence of points $x_1,x_2,\ldots$ such that
\begin{enumerate} 
\item $\lim_{i\to \infty }x_i=  x$;
\item  no  vector  $x_i-x$ lies in   $\mathbb{R}u_1+\cdots+\mathbb{R}u_k$;
\item  defining  ${x_i^1= {(x_i-x)}/{||x_i-x||}}$ and inductively,  
\[x_i^{l}=\frac{x_i-x -\mathsf{p}_{l-1}(x_i-x)}{||x_i-x -\mathsf{p}_{l-1}(x_i-x)||}\,\,\,\,\,  \mbox{ $(l\leq k)$},\]
one has  $\lim_{i\to \infty }x^s_i =  u_s$, for each $s\in\{1,\ldots,k\}$.
\end{enumerate}
A $k$-tangent  $u=(u_1,\ldots,u_k)$  of 
 $X\seq \twiddle{\I}^n$    at   $x$ 
   is said to be \word{rationally outgoing}
 if  there is a rational simplex $S$,
  together with a
face $F\seq S$  and
  a $k$-tuple
  $\lambda=(\lambda_1,\ldots,\lambda_k)\in\mathbb{R}_{>0}$  
  such that
\begin{enumerate}
\item $S\supseteq {\rm conv}(x,x+\lambda_1u_1,\ldots, x+\lambda_1u_1+\cdots+\lambda_ku_k)$;
\item $F\not\supseteq  {\rm conv}(x,x+\lambda_1u_1,\ldots, x+\lambda_1u_1+\cdots+\lambda_ku_k)$;
\item $F\cap X = S\cap X.$
\end{enumerate}
\end{definition}
We now recall the main result \cite[Theorem 2.4]{Cab201X}.

\begin{theorem}\label{theorem:aereo}
For any  closed set  $X\seq \twiddle{\I}^n$ the following conditions are equivalent:
\begin{enumerate}[(i)]
\item The MV-algebra $\McN(X)$ is strongly semisimple.
\item The set $X$ has no rationally outgoing $k$-tangent, for $k=1,\ldots,n-1$.
\end{enumerate}
\end{theorem}

The duality between $\MVS$ and $\TS$ can be used to extend the above characterisation to any strongly semisimple MV-algebra.  We start by noticing the following crucial fact.

\begin{lemma}\label{l:StSeSi1}
An MV-algebra $\A$ is strongly semisimple if, and only if, each of its finitely generated subalgebras are strongly semisimple.
\end{lemma}
\begin{proof}
It follows from \cite[Proposition~4.1]{BuMu201X}.
\end{proof}

\begin{proposition}\label{p:StSeSi2}
Given an MV-algebra  $A$, the following statements are equivalent:
\begin{enumerate}[(i)]
\item\label{p:StSeSi:item2} Each finitely generated subalgebra of $\A$ is strongly semisimple.
\item\label{p:StSeSi:item3} The algebra $\A$ is semisimple and for each finite $C\seq A$, the set $\pi^{A}_{C}(\FW(\A))\seq \twiddle{\I}^C$
does not have rationally outgoing $k$-tangents for any $k$.
\end{enumerate}
\end{proposition}
\begin{proof}
\eqref{p:StSeSi:item2}$\Rightarrow$\eqref{p:StSeSi:item3} We first notice that by Lemma \ref{p:StSeSi2}, $\A$ is strongly semisimple, hence semisimple.  Now, let $C$ be a finite subset of $A$ and $\B$  the subalgebra of $\A$ generated by $C$. By Lemma~\ref{Lem:useful}, $\pi^{A}_{C}(\FW(\A))\seq \twiddle{\I}^C$ is $\Z$-homeomorphic to $\FW(\B)$. By hypothesis $\B$ is strongly semisimple, so by Theorem~\ref{theorem:aereo}, $\FW(\B)\cong\pi^{A}_{C}(\FW(\A))$ has no rationally outgoing $k$-tangent. Hence \eqref{p:StSeSi:item3} is proved.

\eqref{p:StSeSi:item3}$\Rightarrow$\eqref{p:StSeSi:item2} Let $C$ finite subset of $A$ and $\B$  the subalgebra of $\A$ generated by $C$. Combining Lemma~\ref{Lem:useful} and Remark~\ref{Rem}, we have that $\pi^{A}_{C}(\FW(\A))\seq \twiddle{\I}^C$ is $\Z$-homeomorphic to $\FW(\B)$. Then by Theorem~\ref{theorem:aereo}, $\B$ is strongly semisimple if, and only if, $\pi^{A}_{C}(\FW(\A))$ has no rationally outgoing $k$-tangent. We conclude that \eqref{p:StSeSi:item2} and \eqref{p:StSeSi:item3} are equivalent.
\end{proof}

\begin{corollary}\label{cor:StSeSi}
An MV-algebra $A$ is strongly semisimple if, and only if, $\A\cong\McN(X)$ for a closed set $X\seq \twiddle{\I}^A$ such that, whenever $C\seq A$ is finite, the set $\pi^{A}_{C}(X)\seq \twiddle{\I}^C$ does not have rationally outgoing $k$-tangents for any $k$.
\end{corollary}

%%%%%%%%%%%%%%%%%%%%%%%%%%%
\subsection{Polyhedral MV-algebras}
\label{sect:polyhedral}
%%%%%%%%%%%%%%%%%%%%%%%%%%%

An MV-algebra $\A$ is said \word{finitely generated polyhedral}  if there exist $n\in\mathbb{N}$ and a polyhedron $P\seq \mathbb{R}^{n}$ such that $\A\cong \McN(P)$. Finitely generated polyhedral MV algebras were introduced in \cite{BusCabMun201X}, were their main properties were proved. By definition polyhedral MV-algebras are finitely generated. In this section, we use the duality between semisimple MV-algebras and Tychonoff spaces to extend the notion of polyhedron and finitely generated polyhedral MV-algebra to the non finitely generated case.  Our definitions are based on the following observation (see also Remark \ref{r:inf-pol}).

\begin{lemma}
Let $\A$ be a finitely generated polyhedral MV-algebra. Then for each finite $C\seq A$, the set $\pi^{A}_{C}(\FW(\A))\seq\twiddle{\I}^C$ is a polyhedron. 
\end{lemma}
\begin{proof}
Since $\A$ is a finitely generated polyhedral MV-algebra, there exist a natural number $n>0$ and a polyhedron $P\in\I^n$ such that $\A\cong\McN(P)$. 
Then $P\cong\FW(\McN(P))\cong\FW(\A)$. Let $\eta\colon P\to \FW(\A)$  be a $\Z$-homeomorphism. Then $\pi^{A}_{C}\circ\eta$ is a $\Z$-map. Since $P$ is a polyhedron $\pi^{A}_{C}\circ\eta(P)=\pi^{A}_{C}(\eta(P))=\pi^{A}_{C}(\FW(\A))$ is a polyhedron, for it is the image of a polyhedron under a piecewise linear map (see e.g., \cite[Corollary 2.5]{rourke2012introduction}).
\end{proof}

\begin{definition}[Infinite dimensional polyhedra]
We say that $X\seq\twiddle{\I}^Y$ is an \word{infinite dimensional polyhedron} if for any finite $Z\seq Y$, $\pi^{Y}_{Z}(X)\seq\twiddle{\I}^Z$ is a (finite dimensional) polyhedron.
Moreover, if each $\pi^{Y}_{Z}(X)\seq\twiddle{\I}^Z$ is a rational polyhedron, we say that $X$ is a  \word{infinite dimensional rational polyhedron}.
We say that an MV-algebra $\A$ is \word{(rationally) polyhedral} if it is isomorphic to $\McN(X)$ for some (rational) polyhedron $X$ (not necessarily finite dimensional).  
\end{definition}
It is readily seen that being an infinite dimensional (rational) polyhedron is preserved by $\Z$-maps.
\begin{theorem}\label{Theo:CharPol}
Given an MV-algebra $\A$ the following are equivalent:
\begin{enumerate}[(i)]
\item\label{Theo:CharPol:item1} $\A$ is polyhedral;
\item\label{Theo:CharPol:item2} $\A$ is semisimple and $\FW(\A)$ is a (infinite dimensional) polyhedron;
\item\label{Theo:CharPol:item3} each finitely generated subalgebra of $\A$ is polyhedral.
\end{enumerate}
\end{theorem}
\begin{proof}
\eqref{Theo:CharPol:item1}$\Leftrightarrow$\eqref{Theo:CharPol:item2}. $\A$ is polyhedral if, and only if, there exists a, possibly infinite dimensional, polyhedron $P$ such that $\A\cong \McN(P)$.  In turn, this is equivalent to the fact that $\A$ is semisimple and $\FW(\A)\cong \FW(\McN(P))\cong P$.

\eqref{Theo:CharPol:item2}$\Rightarrow$\eqref{Theo:CharPol:item3}. Let $\B$ be the subalgebra of $\A$ generated by a finite subset $C\seq A$.  If $\A$ is semisimple, by Lemma \ref{Lem:useful} $\FW(\B)$ is $\Z$-homeomorphic to $\pi^{A}_{C}(\FW(\A))$. Since $\FW(\A)$ is an infinite dimensional polyhedron, $\FW(\B)$ if a polyhedron, hence $\B$ is finite dimensional polyhedral.  We conclude that $\A$ is  polyhedral.

\eqref{Theo:CharPol:item3}$\Rightarrow$\eqref{Theo:CharPol:item1} Observe first that since every singly-generated subalgebra of $\A$ is polyhedral, it is also semisimple by definition.  Therefore $\A$ cannot contain infinitesimal elements, hence it is semisimple.  Consider now, the dual space $\FW(\A)\seq \twiddle{\I}^{A}$.  Let $C$ be a finite subset of $A$ and $\B$ the algebra generated by $C$.  By Lemma \ref{Lem:useful}, $\FW(B)\cong \pi^{A}_{C}(\FW(\A))$. By hypothesis $\B$ is polyhedral, therefore $\pi^{A}_{C}(\FW(\A))$ is a (finite dimensional) polyhedron, and the claim is proved.
\end{proof}

\begin{corollary}
Every polyhedral MV-algebra is strongly semisimple.
\end{corollary}
\begin{proof}
By \cite[Theorem 5.1]{BusCabMun201X}, every finitely generated polyhedral MV-algebra is strongly semisimple. So, combining Corollary~\ref{cor:StSeSi} and Theorem~\ref{Theo:CharPol} we obtain the statement.
\end{proof}

\begin{theorem}
An MV-algebra is  polyhedral if, and only if, it is the direct colimit of finitely generated polyhedral MV-algebras where transition maps are embeddings.  Dually, a compact subspace of $[0,1]^{X}$ is a (infinite dimensional) polyhedron if, and only if, it is the direct limit of finite dimensional polyhedra, with surjective $\Z$-maps among them.
\end{theorem}
\begin{proof}
One direction is a direct consequence of the fact that every algebra is a direct colimit of its finitely generated subalgebras with inclusions as transition maps and Theorem \ref{Theo:CharPol}.  For the other direction, consider an MV-algebra $A$ that is a colimit of a directed diagram of finitely generated polyhedral MV-algebras $(A_{i}, e_{ij}\mid i,j\in I)$ where $e_{ij}\colon A_{i}\to A_{j}$ is an embedding. Consider a finitely generated subalgebra $B$ of $A$, and let $a_{1},...,a_{n}$ its generators. As $a_{1},...,a_{n}$ belong to $A$ and the colimit maps are embeddings, there must be indices $i_{1},...,i_{n}\in I$ such that $a_{{k}}\in A_{i_{k}}$.  Since the diagram is directed, there exists $\bar{i}\in I$ with transition maps $e_{i_{k}\bar{i}}\colon A_{i_{k}}\to A_{\bar{i}}$.  Then, $B$ is a subalgebra of $A_{\bar{i}}$ and the latter is polyhedral by hypothesis.  So, by Theorem \ref{Theo:CharPol}, $B$ is polyhedral. This concludes the proof of the equivalence.
Finally, the second statement is equivalent to the first, modulo the duality of the previous section.
\end{proof}

%\bibliographystyle{spmpsci}
%\bibliography{biblio}

\begin{thebibliography}{10}
\providecommand{\url}[1]{{#1}}
\providecommand{\urlprefix}{URL }
\expandafter\ifx\csname urlstyle\endcsname\relax
  \providecommand{\doi}[1]{DOI~\discretionary{}{}{}#1}\else
  \providecommand{\doi}{DOI~\discretionary{}{}{}\begingroup
  \urlstyle{rm}\Url}\fi

\bibitem{Bo1930}
Bouligand, H.: Sur les surfaces d\'epourvues de points hyperlimites.
\newblock Ann. Soc. Polon. Math. \textbf{9}, 32--41 (1930)

\bibitem{BS1981}
Burris, S., Sankappanavar, H.P.: A course in universal algebra.
\newblock Graduate texts in Mathematics. Springer-Verlag (1981)

\bibitem{BusCabMun201X}
Busaniche, M., Cabrer, L., Mundici, D.: {Polyhedral MV-algebras}.
\newblock Fuzzy Sets and Systems  (in press, doi:10.1016/j.fss.2014.06.015)

\bibitem{BuMu201X}
Busaniche, M., Mundici, D.: {Bouligand-Severi tangents in MV-algebras}.
\newblock Rev. Mat. Iberoamericana \textbf{30}(1), 191--201 (2014)

\bibitem{Cab201X}
Cabrer, L.M.: {Bouligand-Severi $k$-tangents and strongly semisimple
  MV-algebras}.
\newblock Journal of Algebra \textbf{404}, 271--283 (2014)

\bibitem{cabrer2006priestley}
Cabrer, L.M., Celani, S.: Priestley dualities for some lattice-ordered
  algebraic structures, including {MTL}, {IMTL} and {MV}-algebras.
\newblock Open Mathematics \textbf{4}(4), 600--623 (2006)

\bibitem{Caramello:2014aa}
Caramello, O., Marra, V., Spada, L.: General affine adjunctions, Nullstellens{\"a}tze, and dualities  (2014).
\newblock \urlprefix\url{http://arxiv.org/abs/1412.8692}

\bibitem{CigDotMun2000}
Cignoli, R., D'Ottaviano, I., Mundici, D.: Algebraic Foundations of Many-valued
  Reasoning, \emph{Trends in Logic, Studia Logica Library}, vol.~7.
\newblock Kluwer Academic (2000)

\bibitem{cignoli2004extending}
Cignoli, R., Dubuc, E.J., Mundici, D.: Extending Stone duality to multisets and
  locally finite MV-algebras.
\newblock Journal of pure and applied algebra \textbf{189}(1), 37--59 (2004)

\bibitem{NatDual88}
Clark, D.M., Davey, B.A.: Natural dualities for the working algebraist,
  \emph{Cambridge Studies in Advanced Mathematics}, vol.~57.
\newblock Cambridge University Press, Cambridge (1998)

\bibitem{cornish1977chinese}
Cornish, W.: The chinese remainder theorem and sheaf representations.
\newblock Fundamenta Mathematicae \textbf{3}(96), 177--187 (1977)

\bibitem{piggy}
Davey, B.A., Haviar, M., Priestley, H.A.: Piggyback dualities revisited.
\newblock Algebra Universalis  (to appear)  \urlprefix\url{http://arxiv.org/1501.02512}

\bibitem{NatDualChap}
Davey, B.A., Werner, H.: Dualities and equivalences for varieties of algebras.
\newblock In Contributions to Lattice Theory (Szeged, 1980). Huhn, A.P. and
  Schmidt, E.T. (eds.) pp. 101--275 (1983)

\bibitem{DuPo2010}
Dubuc, E., Poveda, Y.: {Representation theory of MV-algebras}.
\newblock Annals of Pure and Applied Logic  \textbf{161}(8), 1024--1046 (2010)

\bibitem{Ew1996}
Ewald, G.: Combinatorial convexity and algebraic geometry, \emph{Grad. Texts in
  Math.}, vol. 168.
\newblock Springer-Verlag, New York (1996)

\bibitem{filipoiu1995compact}
Filipoiu, A., Georgescu, G.: Compact and Pierce representations of MV-algebras.
\newblock Revue Roumaine de Mathematiques Pures et Appliquees \textbf{40}(7),
  599--618 (1995)

\bibitem{gehrke2014sheaf}
Gehrke, M., van Gool, S.J., Marra, V.: Sheaf representations of {MV}-algebras
  and lattice-ordered abelian groups via duality.
\newblock Journal of Algebra \textbf{417}, 290--332 (2014)

\bibitem{MR0422107}
Keimel, K.: The representation of lattice-ordered groups and rings by sections
  in sheaves.
\newblock In: Lectures on the applications of sheaves to ring theory ({T}ulane
  {U}niv. {R}ing and {O}perator {T}heory {Y}ear, 1970--1971, {V}ol. {III}), pp.
  1--98. Lecture Notes in Math., Vol. 248. Springer, Berlin (1971)

\bibitem{McL1998}
Mac~Lane, S.: Categories for the working mathematician, vol.~5.
\newblock Springer Verlag (1998)

\bibitem{MarSpa2013}
Marra, V., Spada, L.: {The dual adjunction between MV-algebras and Tychonoff
  spaces}.
\newblock Studia Logica (Special issue dedicated to the memory of Leo Esakia)
  \textbf{100}(1-2), 253--278 (2012)

\bibitem{MarSpa11}
Marra, V., Spada, L.: Duality, projectivity, and unification in {\L}ukasiewicz
  logic and {MV}-algebras.
\newblock Ann. Pure Appl. Logic \textbf{164}(3), 192--210 (2013).

\bibitem{martinez1990priestley}
Martinez, N.G.: The {P}riestley duality for {W}ajsberg algebras.
\newblock Studia Logica \textbf{49}(1), 31--46 (1990)

\bibitem{martinez1996simplified}
Martinez, N.G.: A simplified duality for implicative lattices and
  $\ell$-groups.
\newblock Studia Logica \textbf{56}(1-2), 185--204 (1996)

\bibitem{mcnaughton1951theorem}
McNaughton, R.: A theorem about infinite-valued sentential logic.
\newblock The Journal of Symbolic Logic \textbf{16}(01), 1--13 (1951)

\bibitem{Mun2011}
Mundici, D.: Advanced {\L}ukasiewicz calculus and MV-algebras, \emph{Trends in
  Logic}, vol.~35.
\newblock Springer Verlag, New York (2011)

\bibitem{niederkorn2001natural}
Niederkorn, P.: Natural dualities for varieties of {MV}-algebras, {I}.
\newblock Journal of mathematical analysis and applications \textbf{255}(1),
  58--73 (2001)

\bibitem{rourke2012introduction}
Rourke, C., Sanderson, B.: Introduction to piecewise-linear topology.
\newblock Springer Science \& Business Media (2012)

\bibitem{Se1927}
Severi, F.: Conferenze di geometria algebrica (Raccolte da B. Segre), pp.
  1927--1930.
\newblock Stabilimento tipo-litografico del Genio Civile and Zanichelli (1927)

\bibitem{Se1931}
Severi, F.: Su alcune questioni di topologia infinitesimale.
\newblock Ann. Soc. Polon. Math. \textbf{9}, 97--108 (1931)

\end{thebibliography}

\end{document}